\documentclass[a4paper,12pt]{article}

\newenvironment{proof}{\noindent {\bf Proof:}}{\hfill $\Box$}

\newtheorem{theorem}{Theorem}
\newtheorem{lemma}{Lemma}

\newtheorem{corollary}{Corollary}

\usepackage{amsmath}
\usepackage{amssymb}
\usepackage{amsfonts}
\usepackage{graphicx}
\usepackage{subfigure}

\usepackage[belowskip = 10.5 pt]{caption}

\title{\bf Inner approximations of the region of attraction for polynomial dynamical systems}

\begin{document}

\author{Milan Korda$^1$, Didier Henrion$^{2,3,4}$, Colin N. Jones$^1$}

\footnotetext[1]{Laboratoire d'Automatique, \'Ecole Polytechnique F\'ed\'erale de Lausanne, Station 9,
CH-1015, Lausanne, Switzerland. {\tt \{milan.korda,colin.jones\}@epfl.ch}}
\footnotetext[2]{CNRS, LAAS, 7 avenue du colonel Roche, F-31400 Toulouse; France. {\tt henrion@laas.fr}}
\footnotetext[3]{Universit\'e de Toulouse, LAAS, F-31400 Toulouse; France}
\footnotetext[4]{Faculty of Electrical Engineering, Czech Technical University in Prague,
Technick\'a 2, CZ-16626 Prague, Czech Republic}

\date{ \today}

\maketitle

\begin{abstract}
In a previous work we developed a convex infinite dimensional linear programming (LP) approach
to approximating the region of attraction (ROA) of polynomial dynamical systems subject to compact
basic semialgebraic state constraints. Finite dimensional relaxations to the infinite-dimensional LP
lead to a truncated moment problem in the primal and a polynomial sum-of-squares problem in the dual.
This primal-dual linear matrix inequality (LMI) problem can be solved numerically with standard
semidefinite programming solvers, producing a hierarchy of outer (i.e. exterior) approximations
of the ROA by polynomial sublevel sets, with a guarantee of almost uniform and set-wise convergence.
In this companion paper, we show that our approach is flexible enough to be modified so as to generate
a hierarchy of polynomial inner (i.e.\,interior) approximations of the ROA with similar convergence
guarantees.
\end{abstract}

\section{Introduction}

\setcounter{footnote}{4}

Given an autonomous nonlinear system and a target set, the region of attraction (ROA)
is the set of all states that end in the target set at a given time without leaving
the state constraint set\footnote{There are various modifications on this setup
(e.g., one may consider asymptotic convergence instead of finite-time reachability,
with or without constraints, in the presence of disturbances and/or uncertainty,
or in a controlled setting); most of these modifications are amenable to the methods
presented in this paper, sometimes with different qualitative results.}.
The ROA is one of the principal sets associated to any dynamical system and goes
by many other names in the literature (e.g., backward reachable set or capture
basin~\cite{aubinViable}).

In~\cite{roa} we showed (in a controlled setting) that there is a genuinely primal
convex characterization of the ROA. Optimization over system trajectories is formulated
as optimization over occupation measures, leading to an infinite dimensional linear programming (LP)
problem in the cone of nonnegative measures. Finite dimensional relaxations of the dual
of this problem then provide a converging sequence of \emph{outer} approximations to the ROA.
For a description of alternative techniques for numerical approximations of the ROA,
please consult \cite{chesi} or \cite{roa} and the many references therein.

In this paper we show, within the same measure-theoretic framework, that there exists an infinite dimensional LP whose finite-dimensional relaxations provide a converging sequence of \emph{inner} approximations to the ROA. This paper can therefore be seen as a complement to~\cite{roa}.
To simplify our developments and to emphasize our contribution,
we focus only on the uncontrolled setting. The main idea is to construct a converging sequence of outer approximations to the \emph{complement} of the ROA. There are certain difficulties, topological in nature, associated with this approach. A careful distinction had to be made between trajectories leaving the constraint set and trajectories hitting its boundary. This then translates to a (sometimes subtle, but necessary) distinction
between open and closed semialgebraic sets. Fortunately, the LP formulation proposed in \cite{roa}
was flexible enough to allow for these modifications.

Generally speaking, and consistently with our previous work \cite{roa}, we believe that
the main virtues of our approach are overall convexity, conceptual simplicity and compactness.
Both primal and dual finite-dimensional relaxations turn out to be linear matrix inequalities (LMI),
also called semidefinite programming (SDP) problems, with no tuning parameters besides
the relaxation order and no initialization data besides
the defining ingredients of the problem. In addition, the inner approximations obtained
are particularly simple -- they are given by a sublevel set of a single polynomial of a predefined degree.
Therefore, an ROA approximation in analytic form can be readily obtained by solving a single LMI
using freely available software (e.g., SeDuMi~\cite{sedumi}).

\section{Problem statement}
Consider the autonomous system
\begin{equation}\label{sys}
	\dot{x}(t) = f(t,x(t)), \quad x(t) \in X\subset\mathbb{R}^n, \quad t \in [0,T]
\end{equation}
with a given vector field $f$ with polynomial entries $f_i\in {\mathbb R}[t,x]$, $i=1,\ldots,n$, 
final time $T>0$. The state trajectory $x(\cdot)$ is constrained to
a nonempty open\footnote{The requirement of the constraint set being open is merely technical, for
this considerably simplifies the developments and the proofs.} basic semialgebraic set\footnote{For clarity of exposition we consider the constraint set given by a single superlevel set of a polynomial. The approach can, however, be straightforwardly extended to constraint sets defined by the intersection of 
finitely many polynomial superlevel sets.}
\begin{equation}\label{con}
X := \{x \in {\mathbb R}^n \: :\: {g_X}(x) > 0\},
\end{equation}
where the polynomial ${g_X} \in {\mathbb R}[x]$ is such that the set
\[
\bar{X} := \{x \in {\mathbb R}^n \: :\: {g_X}(x) \ge 0\} \supset X
\]
is compact\footnote{Note that the closed semialgebraic set 
$\bar{X} = \{x\: :\: g_X(x) \geq 0\}$ can be strictly larger than the closure
of the open semialgebraic set $X = \{x\: :\: g_X(x)>0\}$, consider in $\mathbb R$ e.g.
$g_X(x)=(1-x^2)(2+x)^2$.
For a similar reason, note also that $X$ bounded does not imply $\bar{X}$ bounded. Indeed, 
in ${\mathbb R}^2$ with $g_X(x) = (1-x^2_1-x^2_2)(2+x_1)^2$ we have $X =  \{x\::\: ||x|| < 1\}$ and $\bar{X} =  \{x\::\: ||x|| \le 1\}\cup \{x\::\: x_1=-2 \}$.}.

The vector field $f$ is polynomial and therefore Lipschitz on the compact set $\bar{X}$. As a result,  for any $x_0\in \bar{X}$ there exists a unique maximal solution $x(\cdot)$ to ODE~(\ref{sys}).
The time interval on which this solution is defined contains the time interval on which $x(t)\in \bar{X}$.

\subsection{Region of attraction (ROA)}
Given a final time $T$ and an open bounded basic semialgebraic target set
\[
X_T := \{x \in {\mathbb R}^n \: :\: {g_T}(x) > 0\} \subset X,
\]
the region of attraction (ROA) is defined as
\begin{equation}\label{eq:roa}
	X_0 := \Big\{x_0\in X \: : \: \exists\, x(\cdot)\; \mathrm{s.t.}\; \dot{x}(t) = f(t,x(t)),\; x(0)=x_0,\; x(T)\in X_T,  \:x(t)\in X\:\:, \forall t\,\in [0,T] \Big\}.
\end{equation}
In words, the ROA is the set of all initial states from $X$ for which the unique solution to~(\ref{sys}) stays in $X$ for all $t\in [0,T]$ and ends in the target set at time $T$.

%
%

\subsection{Complement ROA}
The idea to get \emph{inner} approximations of the ROA $X_0$ is to construct \emph{outer} approximations of the complement ROA $X_0^c := X\setminus X_0$. By continuity of solutions to~(\ref{sys}), the set $X_0^c$ is equal to
\begin{align*}
X_0^c = \big\{x_0\in X \: : \: &\exists\,x(\cdot)\;\text{s.t.}\; \dot{x}(t)=f(t,x(t))\; \text{and}\;\\
 & \exists\, t\in [0,T]\; \mathrm{s.t.}\; x(t)\in X_\partial \ \text{and/or}\ x(T)\in X_T^c\big\},
\end{align*}
where
\[
 X_T^c := \{x\in\mathbb{R}^n \: : \: g_X(x)\ge 0,\; g_T(x)\le0\}
\]
is the complement of $X_T$ in $X$ and
\[
 X_\partial := \{x\in\mathbb{R}^n\::\: g_X(x) = 0 \}.
\]
In words, $X_0^c$ is the set of initial states that give rise to trajectories which do not end up in $X_T$ at time $T$ \emph{and/or} violate the state constraint at some point between $0$ and $T$.

\section{Occupation measures}
In this section we introduce the concept of occupation measures and show how the nonlinear system dynamics can be equivalently described by a linear equation on measures.

\paragraph{Notation} We will use the following notation. The vector space of all signed Borel measures with support contained in a Borel set $K$ is denoted by $M(K)$. The support (i.e., the smallest closed set whose complement has a zero measure) of a measure $\mu$ is denoted by $\mathrm{spt}\,\mu$. The space of continuous functions on $K$ is denoted by $C(K)$ and likewise the space of continuously differentiable functions is $C^1(K)$. The indicator function of a set $K$ (i.e., the function equal to one on $K$ and zero otherwise) is denoted by $I_K(\cdot)$. The symbol $\lambda$ denotes the $n$-dimensional Lebesgue measure (i.e.,  the standard $n$-dimensional volume). The integral of a function $v$ w.r.t a measure $\mu$ over a set $K$ is denoted by $\int_K v(x)\,d\mu(x)$. Sometimes we for simplicity omit the integration variable and/or the set over which we integrate if they are obvious from the context.
 
Now assume $x_0\in \bar{X}$
and define the first hitting time of $X_\partial$ as
\begin{equation}\label{eq:tau}
\tau(x_0) := \min\big\{T, \inf\{t\ge 0 \: : \: x(t\mid x_0)\in X_\partial  \}\big\},
\end{equation}
where $x(\cdot\mid x_0)$ denotes the unique trajectory starting from $x_0$ (which is well defined on the time interval $[0,\tau(x_0)]$). Then we define the \emph{occupation measure} associated to the trajectory starting from $x_0$ by
\[
\mu(A\times B \mid x_0) := \int_0^{\tau(x_0)}\hspace{-0.5em}I_{A\times B}(t,x(t))\,dt
\]
for all Borel\footnote{For brevity we drop the adjective ``Borel'' in the sequel.} sets $A\times B\subset [0,T]\times \bar{X}$.  The interpretation is that the occupation measure measures the time spent by the trajectory $x(\cdot\mid x_0)$ in subsets of the state space.

The occupation measure enjoys the following important property: for any measurable function $v(t,x)$ the equality
\begin{equation}\label{eq:occupProperty}
\int_0^{\tau(x_0)} v(t,x(t))\,dt = \int_{[0,T]\times \bar{X}} \hspace{-1em}v(t,x)\,d\mu(t,x\mid x_0)
\end{equation}
holds. In words, the time integral of a function evaluated along the trajectory $x(\cdot \!\mid\! x_0)$ is equal to the integral of the function w.r.t. the occupation measure associated to $x_0$. Therefore, loosely speaking, all information about the trajectory $x(\cdot \!\mid\! x_0)$ is encoded by the occupation measure $\mu(\cdot\mid x_0)$.

Now suppose that the initial state is not a single point but that its spatial distribution is given by an \emph{initial measure} $\mu_0\in M(\bar{X})$. Then we define the \emph{average occupation measure} $\mu\in M([0,T]\times\bar{X})$ as
\[
\mu(A\times B) := \int_{\bar{X}}\mu(A\times B\mid x_0)\,d\mu_0(x_0).
\]
Lastly, we define the final measure $\mu_T\in M([0,T]\times \bar{X})$ by
\[
\mu_T(B) := \int_{\bar{X}} I_B(x(T\mid x_0))\,d\mu_0(x_0).
\]

To derive an equation linking the three principal measures, consider a test function $v\in C^1([0,T]\times \bar{X})$ evaluated along a trajectory. Using the chain rule and equation~(\ref{eq:occupProperty}) we obtain
\begin{align*}
v\big(\tau(x_0),\, x(\tau(x_0)\mid x_0)\big) - v(0,x_0) &= \int_0^{\tau(x_0)} \frac{d}{dt}v(t,x(t\mid x_0))\,dt \\ 
&= \int_0^{\tau(x_0)} \left(\frac{\partial{v}}{\partial t} + \mathrm{grad}\, v\cdot f(t,x(t\mid x_0))\right)\,dt\\
& = \int_{[0,T]\times\bar{X}} \left(\frac{\partial{v}}{\partial t} + \mathrm{grad}\, v\cdot f(t,x)\right)\,d\mu(t,x\mid x_0)\\
& = \int_{[0,T]\times\bar{X}} \mathcal{L}v(t,x)\,d\mu(t,x\mid x_0)
\end{align*}
where the linear operator $\mathcal{L}:C^1([0,T]\times\bar{X})\to C([0,T]\times\bar{X})$ is defined by
\begin{equation*}
v\mapsto \mathcal{L}v:= \frac{\partial v}{\partial t} + \mathrm{grad}\,v\cdot f.
\end{equation*}
Integrating the above equation w.r.t. $\mu_0$ leads to the equation
\begin{equation}\label{eq:Liouville}
\int_{[0,T]\times \bar{X}}\hspace{-1em}v(t,x)\,d\mu_T(t,x) - \int_{\bar{X}}v(0,x)\,d\mu_0(x) = \int_{[0,T]\times\bar{X}} \hspace{-1em}\mathcal{L}v(t,x)\,d\mu(t,x)\quad \forall\: v\in C^1([0,T]\times\bar{X}),
\end{equation}
which is a linear equation linking the measures $\mu_0$, $\mu$ and $\mu_T$. Equation~(\ref{eq:Liouville}) is sometimes referred to as Liouville's equation.

\section{Primal LP}
In this section we follow the approach developed in~\cite{roa} and derive an infinite-dimensional linear programming (LP) characterization of the complement ROA $X_0^c$. Certain sublevel sets of feasible solutions to the dual of this LP then yield inner approximations to the ROA $X_0$.

The basic idea is to maximize the mass of the initial measure $\mu_0$
under the constraint that it is dominated by the Lebesgue measure, i.e., $\mu_0 \le \lambda$. System dynamics is captured by Liouville's equation~(\ref{eq:Liouville}) and state and terminal constraints
are handled by suitable constraints on the support of the measures. The key idea is then to split the final measure in two measures such that each measure is supported on a suitable compact basic semialgebraic set. More explicitly, we let
\[
\mu_T := \mu_T^1 + \mu_T^2
\] 
with $\mu_T^1\in M([0,T]\times X_\partial)$ and $\mu_T^2\in M(\{T\}\times X_T^c)$. That is, we require that $\mathrm{spt}\,\mu_T^1\subset [0,T]\times X_\partial$ and $\mathrm{spt}\,\mu_T^2\subset \{T\}\times X_T^c$. The interpretation is that measure $\mu_T^1$ models the trajectories that leave $X$, whereas measure $\mu_T^2$ models the trajectories that end in $X_T^c$ (i.e., not in $X_T$). These support constraints on the final measure(s) along with system dynamics enforce that the support of the initial measure $\mu_0$ must be contained in $X_0^c$. Since there are no other constraints on $\mu_0$ besides $\mu_0\le \lambda$, maximization of its mass should yield the restriction of the Lebesgue measure $\lambda$ to $X_0^c$. 

The constraint $\mu_0 \le \lambda$ can be rewritten equivalently as $\mu_0 + \hat{\mu}_0 = \lambda$ for some nonnegative slack measure $\hat{\mu}_0\in M(X)$. This is equivalent to requiring that $ \int w \,d\mu_0 + \int w \,d\hat{\mu}_0 = \int w\,d\lambda$ for all test functions $w\in C(\bar{X})$. In addition, we can drop the time argument from the definition of $\mu_T^2$ since its time component is supported on a singleton.  

This leads to the following optimization problem:
\begin{equation}\label{rrlp}
\begin{array}{rclll}
 p^* &= & \sup & \int1\,d\mu_0 \\
&& \mathrm{s.t.} & \int v\,d\mu_T^1 + \int v(T,\cdot)\,d\mu_T^2 - \int v(0,\cdot)\,d\mu_0 = \int \mathcal{L}v\, d\mu\; &\forall\, v\in C^1([0,T]\times\bar{X}) \\
&&& \int w \,d\mu_0 + \int w \,d\hat{\mu}_0 = \int w\,d\lambda\; &\forall\,w\in C(\bar{X}) \\
&&& \mu_0\geq 0, \: \mu\geq 0,\: \mu_T^1\geq 0,\: \mu_T^2\geq 0,\:\hat{\mu}_0\geq 0\\
&&& \mathrm{spt}\:\mu \subset [0,T]\times \bar{X}, \:\: \mathrm{spt}\:\mu_0 \subset \bar{X},\:\: \mathrm{spt}\:\hat{\mu}_0 \subset \bar{X}\\
&&& \mathrm{spt}\:\mu_T^1 \subset [0,T]\times X_\partial , \:\:\mathrm{spt}\:\mu_T^2 \subset X_T^c,
\end{array}
\end{equation}
where the supremum is over the vector of nonnegative measures \[
(\mu_0,\mu,\mu_T^1,\mu_T^2,\hat{\mu}_0) \in  M(\bar{X})\times M([0,T]\times \bar{X}) \times M([0,T]\times X_\partial)\times M(X_T^c)\times M(\bar{X}).
\] 
Problem~(\ref{rrlp}) is an infinite-dimensional LP in the cone of nonnegative measures. Indeed, the objective is linear, the first two constraints are linear equality constraints and the remaining constraints are conic constraints (the set of nonnegative measures supported on a given set is a positive cone in the vector space of all measures supported on the same set).

The discussion leading to problem~(\ref{rrlp}) is formalized in the following result.
\begin{theorem}\label{thm:1}
The optimal value of LP problem (\ref{rrlp}) is equal to the volume of the complement ROA $X_0^c$, that is, $p^*=\lambda(X_0^c)$.
Moreover, the supremum is attained by the restriction of the Lebesgue measure to the complement ROA $X_0^c$.
\end{theorem}
\begin{proof} Closely follows arguments in \cite{roa}. By definition of the relaxed complement ROA, the unique trajectory $x(\cdot)$ associated to any initial condition $x_0\in X_0^c$ either hits $X_\partial$ at some $t\in [0,T]$ or ends in $X_T^c$. Therefore for any initial measure $\mu_0 \le \lambda$ with $\mathrm{spt}\,\mu_0 \subset X\subset \bar{X}$ there exist an occupation measure $\mu$, final measures $\mu_T^1$, $\mu_T^2$ and a slack measure $\hat{\mu}_0$ such that the constraints of problem~(\ref{rrlp}) are satisfied. One such measure $\mu_0$ is the restriction of the Lebesgue measure to $X_0^c$, and therefore $p^*\ge \lambda(X_0^c)$.

Now we show that $p^* \le \lambda(X_0^c)$. Take a vector of measures $(\mu_0,\mu,\mu_T^1,\mu_T^2,\hat{\mu}_0)$ feasible in~(\ref{rrlp}) and suppose that $\lambda(\mathrm{spt}\,\mu_0 \setminus X_0^c) > 0$. Since any level set of a polynomial has a zero Lebesgue measure we have $\lambda(X_\partial)=0$ and
 \[ 
 \lambda(\mathrm{spt}\,\mu_0 \setminus (X_0^c\cup X_\partial)) =  \lambda(\mathrm{spt}\,\mu_0 \setminus X_0^c) > 0.
 \]
By a superposition principle~\cite[Theorem~3.2]{ambrosio} using arguments of \cite[Appendix~A, Lemma~4]{roa}, there exists a family of admissible trajectories of the ODE~(\ref{sys}) starting from $\mu_0$ generating the occupation measure $\mu$ and the final measure $\mu_T = \mu_T^1+\mu_T^2$. However, this is a contradiction since $\mathrm{spt}\,\mu_0 \setminus (X_0^c\cup X_\partial)\subset X_0$, which means that all trajectories starting from $\mathrm{spt}\,\mu_0 \setminus (X_0^c\cup X_\partial)$ neither hit $X_\partial$ nor end in $\bar{X}_T^c$. Thus, $\lambda(\mathrm{spt}\,\mu_0 \setminus X_0^c) = 0$ and so $\lambda(\mathrm{spt}\,\mu_0) \le \lambda(X_0^c)$. Combining this with the constraint $\mu_0 \le \lambda$ we get $\mu_0(X) = \mu_0(\mathrm{spt}\, \mu_0) \le \lambda(\mathrm{spt}\, \mu_0) \le \lambda(X_0^c)$ for any feasible $\mu_0$.  Therefore $p^*\le \lambda(X_0^c)$ and thus in fact $p^* = \lambda(X_0^c)$.
\end{proof}


\section{Dual LP}
In this section we derive a dual LP on continuous functions, prove the absence of a duality gap between the primal and dual LPs and relate feasible solutions to the dual to the indicator function of the complement ROA $X_0^c$.

By standard infinite-dimensional LP theory (see, e.g., \cite{anderson}), the dual to LP~(\ref{rrlp}) reads
\begin{equation}\label{vlp}
\begin{array}{rclll}
d^* & = & \inf & \int_{X} w(x)\, d\lambda(x) \\
&& \mathrm{s.t.} & \mathcal{L}v(t,x) \leq 0, \:\: &\forall\, (t,x,u) \in [0,T]\times \bar{X} \\
&&& w(x) \ge v(0,x) + 1, \:\: &\forall\, x \in \bar{X} \\
&&& v(T,x) \geq 0, \:\: &\forall\, x \in X_T^c\\
&&& v(t,x) \geq 0,\:\: & \forall (t,x)\in[0,T]\times X_\partial\\
&&& w(x) \geq 0, \:\: &\forall\, x \in \bar{X},
\end{array}
\end{equation}
where the infimum is over $(v,w) \in C^1([0,T]\times \bar{X})\times C(\bar{X})$. 

The intuition is that given $x_0\in X_0^c$ the constraint $\mathcal{L}v\le 0$ forces $v$ to decrease along trajectories as long as it does not hit $X_\partial$ or end in $X_T^c$. Because of the constraint $v\ge 0$ on $[0,T]\times X_\partial \cup \{T\}\times X_T^c$ we must have $v(0,\cdot)\ge 0$ on $X_0^c$. Consequently, $w(x) \ge 1$ on $X_0^c$. This instrumental observation is formalized in the following Lemma.

\begin{lemma}\label{lem:v0}
If $\mathcal{L}v \leq 0$ on $[0,T]\times \bar{X}$, $v \ge 0$ on $([0,T]\times X_\partial)\cup(\{T\}\times X_T^c)$ and $w \ge v(0,\cdot)+1$ on $X$, then $w \geq 1$ on $X_0^c$.
\end{lemma}
\begin{proof}
Take $x_0\in X_0^c$ and consider the first hitting time of $X_\partial$, $\tau:=\tau(x_0)$, defined by~(\ref{eq:tau}). By definition of $X_0^c$ the trajectory starting from $x_0$ will either hit $X_\partial$ or end in $X_T^c$. Therefore $x(\tau) \in ([0,T]\times X_\partial)\cup(\{T\}\times X_T^c)$ and $x(t)\in X$ for $t\in[0,\tau]$. Therefore $v(\tau,x(\tau))\ge 0$, $\mathcal{L}v(t,x(t))\le0$, $\forall\, t\in [0,\tau]$ and so
\[
  0 \le v(\tau,x(\tau)) = v(0,x_0) + \int_0^\tau \mathcal{L}v(t,x(t))\,d t \le v(0,x_0)\le w(x_0) -1.
\] 
\end{proof}

The following result is of key importance for subsequent developments.
\begin{theorem}\label{thm:noGap}
There is no duality gap between primal LP problems (\ref{rrlp}) on measures
and dual LP problem (\ref{vlp}) on functions, that is,
$p^*=d^*$.
\end{theorem}
\begin{proof}
Here we only outline the basic steps; for a detailed argument in a similar setting see~\cite[Theorem~2]{roa}. Since the supports of all measures are compact, the initial measure is dominated by the Lebesgue measure and the final time is finite, we have $\mu_0(\bar{X})\le \lambda(\bar{X}) < \infty$, $\mu_T([0,T]\times \bar{X}) = \mu_0(\bar{X}) < \infty$ and $\mu([0,T]\times X) \le T\mu_T([0,T]\times \bar{X}) < \infty$, where the last two inequalities follow by plugging in $v(t,x) = 1$ and $v(t,x) = t$ in Liouvillel's equation~(\ref{eq:Liouville}). Therefore $p^* < \infty$ and the feasible set of problem~(\ref{rrlp}) is weakly-* bounded. Furthermore, the feasible set of~(\ref{rrlp}) is nonempty since $(\mu_0,\mu,\mu_T^1,\mu_T^2,\hat{\mu}_0) = (0,0,0,0,\lambda)$ is a trivial feasible point; therefore $0\le p^*<\infty$. The absence of a duality gap then follows from \cite[Theorem 3.10]{anderson} using
Alaoglu's theorem (see, e.g., \cite[Chapter 5]{luenberger}) and the weak-* continuity of the adjoint of the operator $\mathcal{L}$. 
\end{proof}

Next, we establish our first convergence\footnote{Please refer to \cite{roa} or, e.g., \cite{ash} for definitions of
the various types of convergence relevant in this context.} result.

\begin{theorem}\label{thm:2}
There is a sequence of feasible solutions to problem (\ref{vlp}) such that its\newline $w$-component converges from above to $I_{X_0}$ in $L^1$ norm and almost uniformly.
\end{theorem}
\begin{proof}
	Follows by the same arguments as Theorem~3 in \cite{roa}.
\end{proof}

\section{LMI relaxations}
In this section we derive finite dimensional semidefinite programming (SDP) or linear matrix inequality (LMI) relaxations to the infinite dimensional LPs~(\ref{rrlp}) and (\ref{vlp}) and establish several convergence results relating these relaxations to the infinite dimensional LPs and to the ROA.

In what follows, $\mathbb{R}_k[\cdot]$ denotes the vector space of real multivariate polynomials of total degree less than or equal to $k$.

Derivation of the finite dimensional relaxations is standard and the reader is referred to~\cite[Section~5]{roa} or to the comprehensive reference~\cite{lasserre}; therefore we only highlight the main ideas. First of all, since the supports of all measures feasible in~(\ref{rrlp}) are compact, these measures are determined by their moments, i.e., by integrals of all monomials (which is a sequence of real numbers when indexed in, e.g., the canonical monomial basis). Therefore, it suffices to restrict the test functions $w(x)$ and $v(t,x)$ in (\ref{rrlp}) to all monomials, reducing the linear equality constraints of~(\ref{rrlp}) to linear equality constraints on the moments. Next, by the celebrated Putinar Positivstellensatz (see~\cite{lasserre,putinar}), the constraint that the support of a measure is included in a given compact basic semialgebraic set is equivalent to the feasibility of an infinite sequence of LMIs involving the so-called moment and localizing matrices, which are linear in the coefficients of the moment sequence. By truncating the moment sequence and taking only the moments corresponding to monomials of total degree less than or equal to $2k$ we obtain a necessary condition for this truncated moment sequence to be the first part of a moment sequence of a measure with the desired support.

This procedure leads to the primal SDP relaxation of order $k$
\begin{equation}\label{plmi}
\begin{array}{rcllll}
p^*_k & =  &\max & (y_0)_0 \\
&& \mathrm{s.t.} & A_k(y,y_0,y_T^1,y_T^2,\hat{y}_0) = b_k \\
&&& M_k(y) \succeq 0, & M_{k-{d_X}_i}({g_X},y)\succeq 0 \\
&&& M_k(y_0) \succeq 0, & M_{k-{d_X}}({g_X},y_0) \succeq 0\\
&&& M_k(y_T^1) \succeq 0, & M_{k-{d_T}}({g_X},y_T^1) \succeq 0 & M_{k-{d_T}}({-g_X},y_T^1) \succeq 0\\
&&& M_k(y_T^2) \succeq 0, & M_{k-{d_T}}({g_X},y_T^2) \succeq 0 & M_{k-{d_T}}({-g_T},y_T^2) \succeq 0\\
&&& M_k(\hat{y}_0) \succeq 0,  & M_{k-d_X}({g_X},\hat{y}_0) \succeq 0\\
&&&  M_{k-1}(t(T-t),y) \succeq 0, &  M_{k-1}(t(T-t),y_T^1) \succeq 0\\
\end{array}
\end{equation}
where the notation $\succeq 0$ stands for positive semidefinite
and the minimum is over moment sequences $(y, y_0, y_T^1,y_T^2,\hat{y}_0)$ truncated to degree $2k$ corresponding to measures $\mu$, $\mu_0$, $\mu_T^1$, $\mu_T^2$ and $\hat{\mu}_0$. The linear equality constraint captures the two linear equality constraints of~(\ref{rrlp}) with $v(t,x)\in\mathbb{R}_{2k}[t,x]$ and $w(x)\in\mathbb{R}_{2k}[x]$ being monomials of total degree less than or equal to $2k$. The matrices $M_k(\cdot)$ are the moment and localizing matrices, following the notations of \cite{lasserre} or \cite{roa}.
In problem~(\ref{plmi}), a linear objective is minimized subject to linear equality constraints and LMI constraints; therefore problem (\ref{plmi}) is an SDP problem.

The SDP problem dual to problem~(\ref{plmi}) turns out to be the sum-of-squares problem
\begin{equation}\label{dlmi}
\begin{array}{rcll}
d^*_k & = & \inf & w \cdot l \\\vspace{0.5mm}
&& \mathrm{s.t.} & -\mathcal{L}v(t,x)
= p(t,x) + q_1(t,x) t(T-t) +  q_2(t,x) {g_X}(x) \\ \vspace{1mm}
&&& w(x)-v(0,x)-1 = p_0(x) + {q_0}_1(x) {g_X}(x) \\\vspace{1mm}
&&& v(t,x) = {p_T}_1(x) + {q_T}_1(t,x) t(T-t) +  r(x) {g_X}(x)  \\
&&& v(T,x) = {p_T}_2(x) +  {q_T}_2(x) {g_X}(x) - {q_T}_3(x) {g_T}(x) \\
&&& w(x) = {s_0}(x) +  {s_1}(x) {g_X}(x),
\end{array}
\end{equation}
where $l$ is the vector of Lebesgue moments over $X$ indexed in the same basis in which the polynomial $w(x)$ with coefficients $w$ is expressed. The minimum is over polynomials $v(t,x)\in\mathbb{R}_{2k}[t,x],$ $w(x)\in\mathbb{R}_{2k}[x]$, over the polynomial $r(x)$ and polynomial sum-of-squares $p(t,x)$, $q_1(t,x)$, $q_2(t,x)$, ${q_0}_1(x)$, ${p_T}_1(x)$, ${p_T}_2(x)$, ${q_T}_1(x)$, ${q_T}_2(x)$, ${q_T}_3(x)$, $s_0(x)$, $s_1(x)$ of appropriate degrees. The constraints that polynomials are sum-of-squares can be written explicitly as LMI constraints~(see, e.g., \cite{lasserre}), and the objective is linear in the coefficients of the polynomial $w(x)$; therefore problem~(\ref{dlmi}) can be formulated as an SDP problem.

\begin{theorem}\label{lem:noGapRelax}
There is no duality gap between primal LMI problem (\ref{plmi}) and dual LMI problem (\ref{dlmi}),
i.e. $p^*_k = d^*_k$.
\end{theorem}
\begin{proof}
Follows by the same arguments based on standard SDP duality theory as Theorem~4 in~\cite{roa}.
\end{proof}

Now we prove our main convergence results.
\begin{theorem}\label{thm:dualConvFun}
Let $w_k \in {\mathbb R}_{2k}[x]$ denote the $w$-component of a solution to the dual LMI problem (\ref{dlmi}) and let $\bar{w}_k(x) =\min_{i\le k} w_i(x)$. Then $1-w_k$ converges from below to $I_{X_0}$ in $L^1$ norm and  $1-\bar{w}_k$ converges from below to $I_{X_0}$ in $L^1$ norm and almost uniformly.
\end{theorem}
\begin{proof}
It follows from Theorem~\ref{thm:2} and from the density of polynomials in the space of continuous functions on compact sets (for a detailed argument in a similar setting see~\cite[Theorem~5]{roa}) that $w_k$ and $\bar{w}_k$ converge from above to $I_{X_0^c}$ in $L_1$ and almost uniformly on $X$, respectively. Therefore $1-w_k$ and $1-\bar{w}_k$ converge from below to $I_{X_0} = 1-I_{X_0^c}$ on $X$ in the same manner.
\end{proof}

The next Corollary follows immediately from Theorem~\ref{thm:dualConvFun}. 
\begin{corollary}\label{cor:pdconv}
The sequence of infima of LMI problems~(\ref{dlmi}) converges monotonically from above to the supremum of the LP problem~(\ref{vlp}), i.e., $d^*\le d_{k+1}^* \le d_k^*$ and $\lim_{k\to\infty} d_k^* = d^*$. Similarly, the sequence of maxima of LMI problems (\ref{plmi}) converges monotonically from above to the maximum of the LP problem (\ref{rrlp}), i.e., $p^* \le p^*_{k+1}\le p_k^*$ and $\lim_{k\to \infty} p^*_k = p^*$.
\end{corollary}
{\bf Proof:} Follows the proof of Corollary~1 in \cite{roa}. Monotone convergence of the dual optima $d_k^*$ follows immediately from Theorem~\ref{thm:dualConvFun} and from the fact that the higher the relaxation order $k$, the looser the constraint set of the minimization problem~(\ref{dlmi}). To prove convergence of the primal maxima observe that from weak SDP duality we have $d_k^* \ge p_k^*$ and from Theorems~\ref{thm:dualConvFun} and~\ref{thm:noGap} it follows that $d_k^* \to d^* = p^*$. In addition, clearly $p_k^* \ge p^*$ and $p_{k+1}^* \le p_k^*$ since the higher the relaxation order $k$, the tighter the constraint set of the maximization problem~(\ref{plmi}). Therefore $p_k^*\to p^*$ monotonically from above. $\Box$

Our last results establishes set-wise convergence of inner approximations to the ROA.

\begin{theorem}
Let $w_k \in {\mathbb R}_{2k}[x]$ denote the $w$-component of a solution to the dual LMI problem (\ref{dlmi})
and let ${X_0}_k := \{x \in X \: :\: w_k(x) < 1\}$. Then ${X_0}_k \subset X_0$,
\[
 \lim_{k\to\infty}\lambda(X_0\setminus X_{0k}) = 0 \quad\text{and}\quad \lambda( X_0\setminus  \cup_{k=1}^{\infty}{X_0}_k) = 0.
\]
\end{theorem}
\begin{proof}
Follows the proof of Theorem~6 in \cite{roa}. From Lemma~\ref{lem:v0} we have $w_k(x)\le 1 \Rightarrow x\in X_0$ for all $x\in X$, and therefore $I_{X_{0k}}\le I_{X_0} $. Since $w_k\ge 0$ on $X$ we also have $1-w_k\le I_{X_{0k}}$ on $X$ and therefore $1-w_k\le I_{X_{0k}}\le I_{X_{0}}$ on $X$. From Theorem~\ref{thm:dualConvFun}, we have $1-w_k \to I_{X_0}$ in $L^1$ norm on $X$. Consequently, 
\begin{align*}
\lambda(X_0) = \int_X I_{X_0}\,d\lambda &\: \:=\: \: \lim_{k\to\infty} \int_X 1-w_k\,d\lambda  \:\le\:  \lim_{k\to\infty} \int_X I_{X_{0k}}\,d\lambda   \:=\:  \lim_{k\to\infty}\lambda(X_{0k})\\ & \: \:\le\: \: \lim_{k\to\infty}\lambda(\cup_{i=1}^k X_{0i}) \: =\: \lambda(\cup_{k=1}^\infty X_{0k}). 
\end{align*}
But since $ X_{0k} \subset X_0 $ for all $k$, we must have
\[
 \lim_{k\to\infty}\lambda(X_{0k}) = \lambda(X_0)\quad \mathrm{and}\quad  \lambda(\cup_{k=1}^\infty X_{0k}) = \lambda(X_0),
\]
which proves the theorem.
\end{proof}

\section{Numerical examples}
In this section we present two numerical examples. The primal problems on measures were modeled using Gloptipoly~3~\cite{glopti} interfaced with the SDP solver SeDuMi~\cite{sedumi}; this solver also returns the solution to the dual SDP relaxation. In Section~\ref{sec:lowOrder} we then investigate how tight low order approximations can be obtained.

\subsection{Univariate cubic dynamics}

\begin{figure*}[th]
\begin{picture}(140,360)
\put(20,200){\includegraphics[width=70mm]{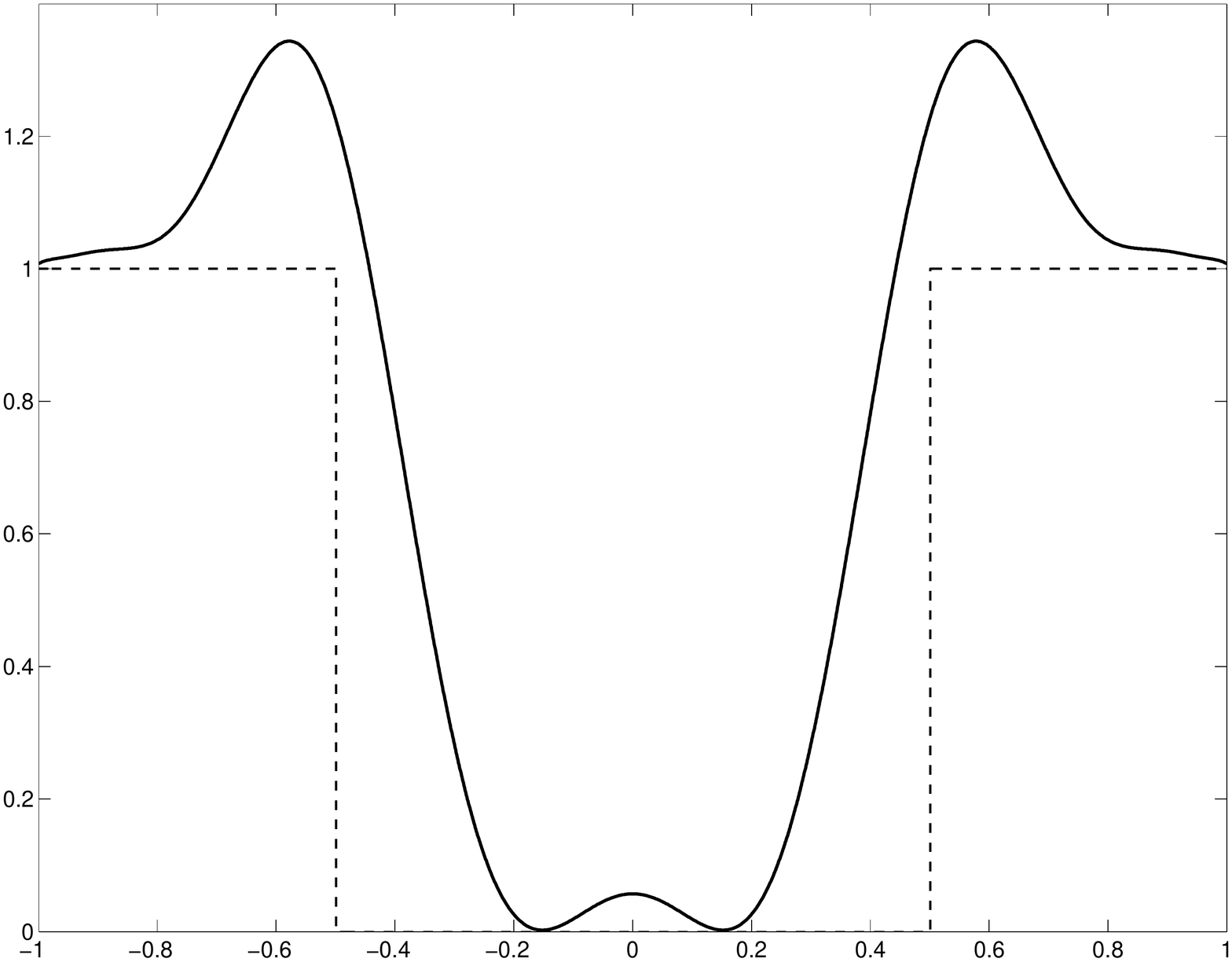}} 
\put(250,200){\includegraphics[width=70mm]{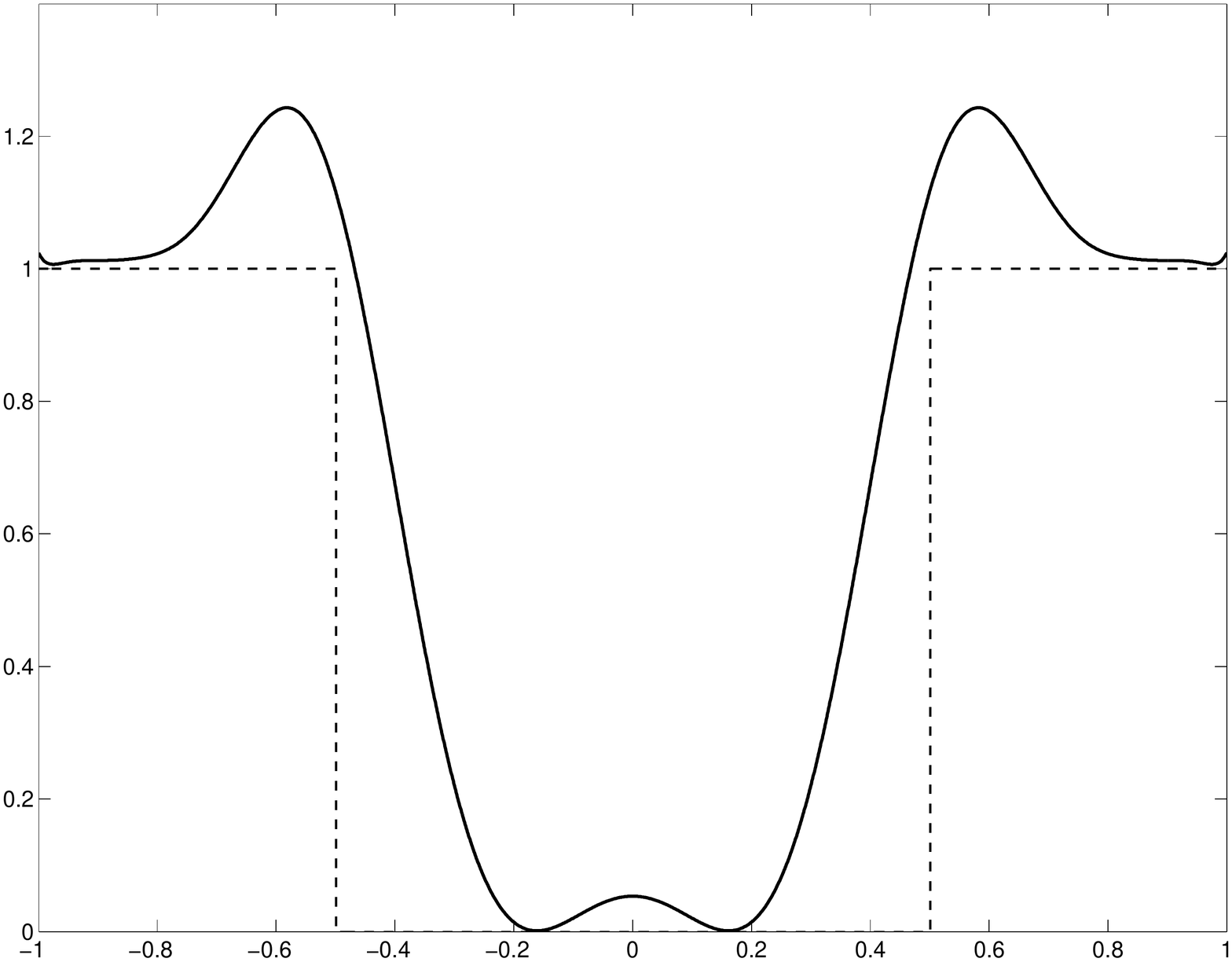}}
\put(20,20){\includegraphics[width=70mm]{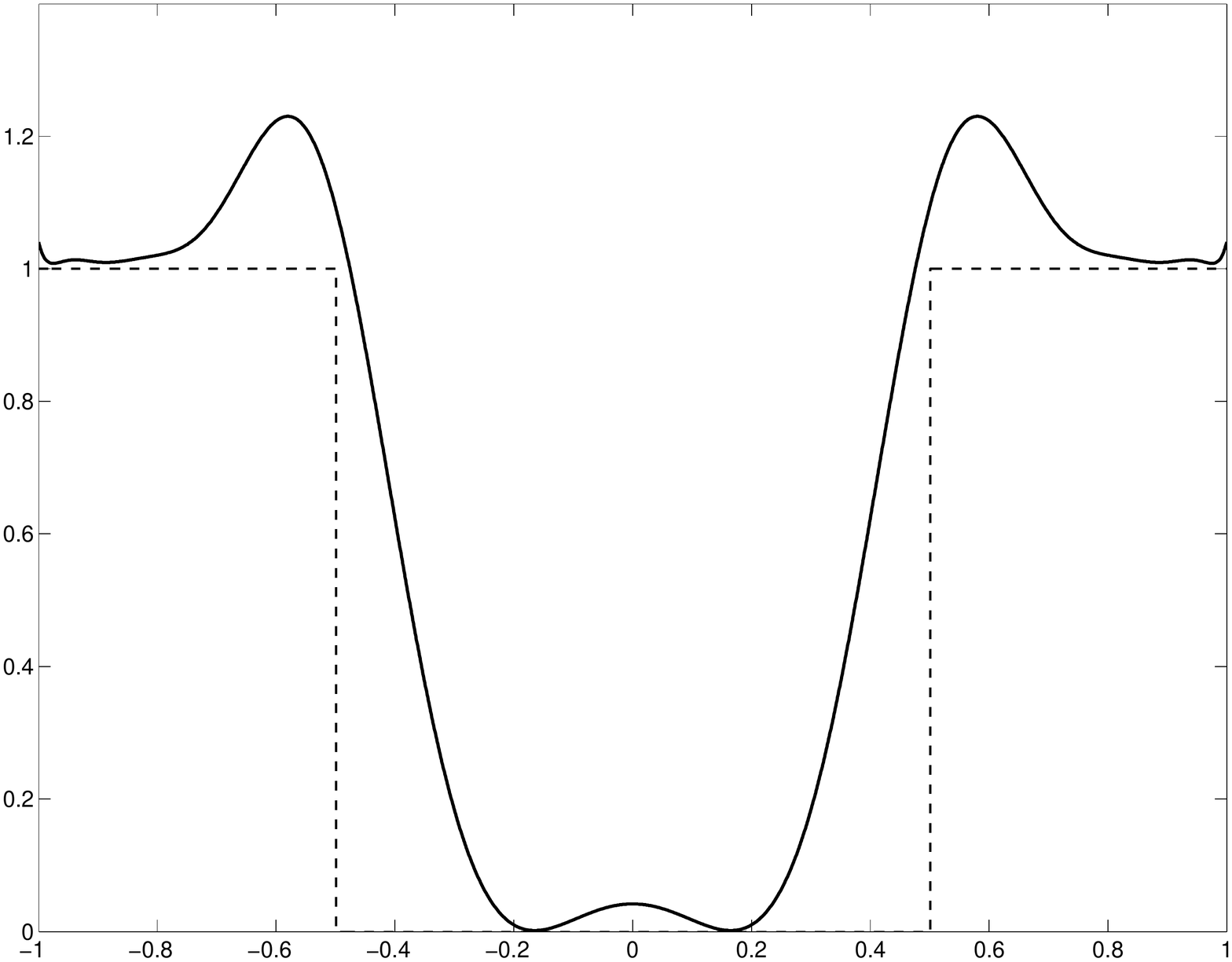}}
\put(250,20){\includegraphics[width=70mm]{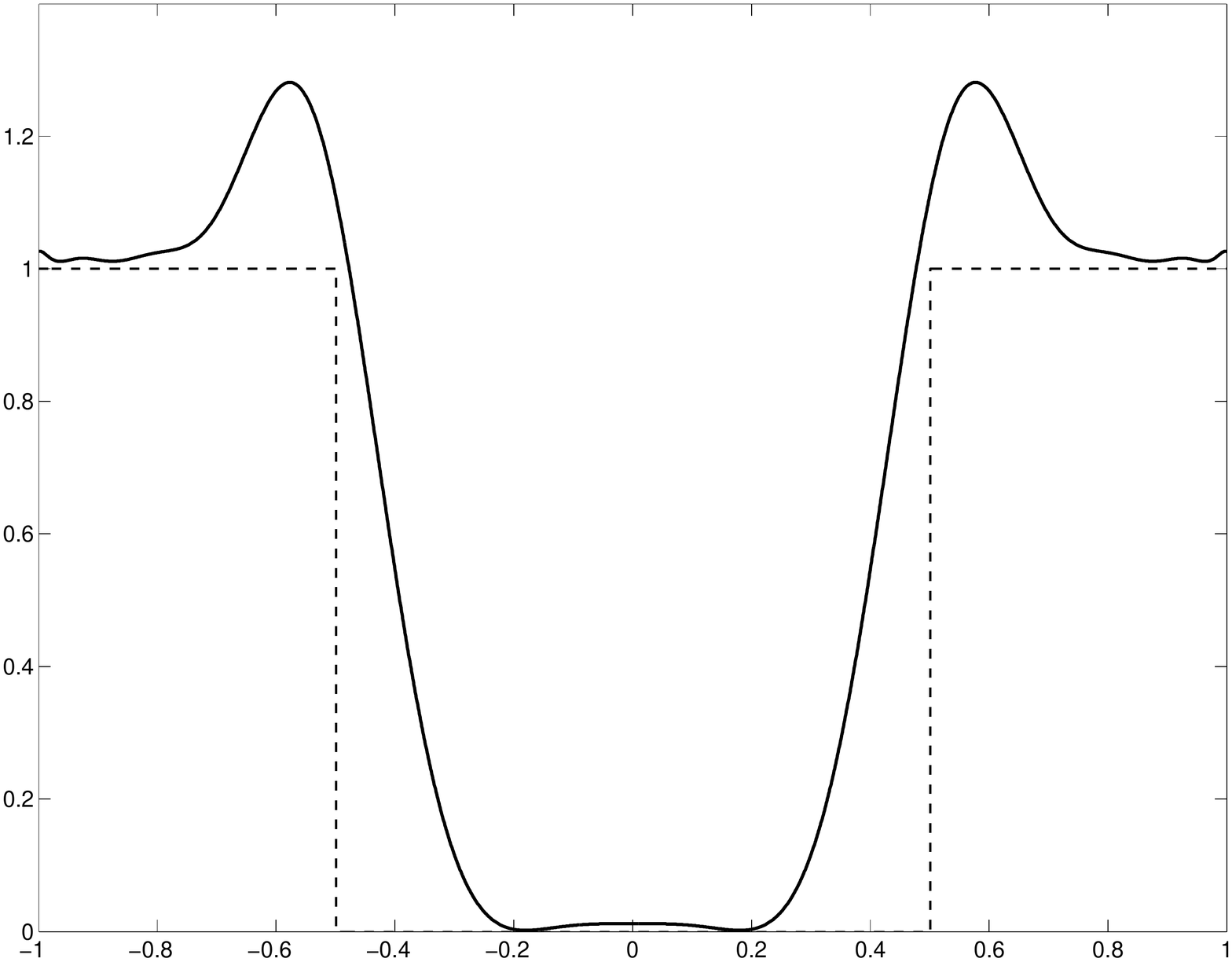}}

\put(120,190){\footnotesize $x$}
\put(350,190){\footnotesize $x$}
\put(120,10){\footnotesize $x$}
\put(350,10){\footnotesize $x$}

\put(110,330){\footnotesize $d = 16$}
\put(340,330){\footnotesize $d=20$}
\put(110,150){\footnotesize $d = 24$}
\put(340,150){\footnotesize $d=28$}

\end{picture}
\caption{Univariate cubic dynamics -- polynomial approximations (solid line) to the complement ROA indicator function $I_{X_0^c} =I_ {[-1,-0.5]}+I_ {[0.5,1]}$ (dashed line) for degrees $d\in\{16,20,24,28\}$.}
\label{fig:1}
\end{figure*}

Consider the system given by
\[\dot{x} = x(x-0.5)(x+0.5), \]
the constraint set $X = [-1,1]$, the final time $T = 10$ and the target set $X_T = [-0.3,0.3]$. The ROA in this setup is $X_0 = [-0.5,0.5]$. Polynomial approximations to the complement ROA for degrees $d\in\{16,20,24,28\}$ are shown in Figure~\ref{fig:1}. As expected, the functional convergence of the polynomial to the discontinuous indicator function is rather slow. A slightly better convergence is observed in the volume error of the sublevel set approximation to the ROA documented in Table~\ref{tab:1}. The relatively slow convergence could be significantly improved if a tighter constraint set $X$ was employed; see Section~\ref{sec:lowOrder} below. Alternative polynomial bases (e.g. Chebyshev polynomials) would also allow tighter higher order
approximations; see \cite{volume} for more details.

\begin{table}[ht]
\centering
\caption{\rm \small Univariate cubic dynamics -- relative error of the inner approximations to the ROA $X_0 = [-0.5,0.5]$ as a function of the approximating polynomial degree.}\vspace{1mm}
\begin{tabular}{c|cccc}\label{tab:1}
degree & 16 & 20 & 24 & 28 \\\hline
error & 11.4\,\% & 6.4\,\% & 4.84\,\% & 4.54\,\% 
\end{tabular}
\end{table}

\subsection{Van der Pol oscillator}

\begin{figure}[h]
  \centering
    \includegraphics[width=110mm]{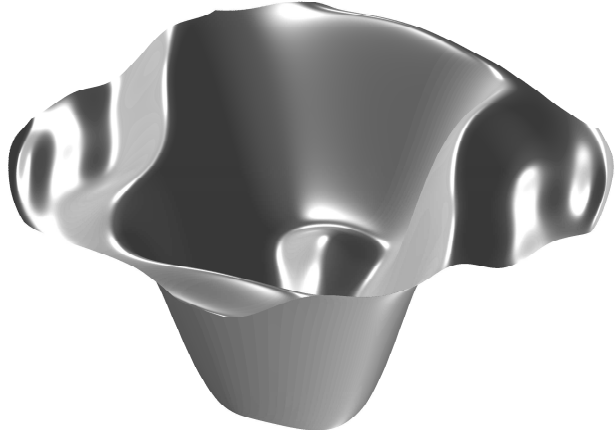}
    \vglue-11mm
    \caption{Van der Pol oscillator  -- degree 18 polynomial approximation to the indicator function of the complement ROA.}
    \label{fig:vp3D}
\end{figure}

\begin{figure*}[ht]
	\begin{picture}(140,360)
	\put(20,200){\includegraphics[width=70mm]{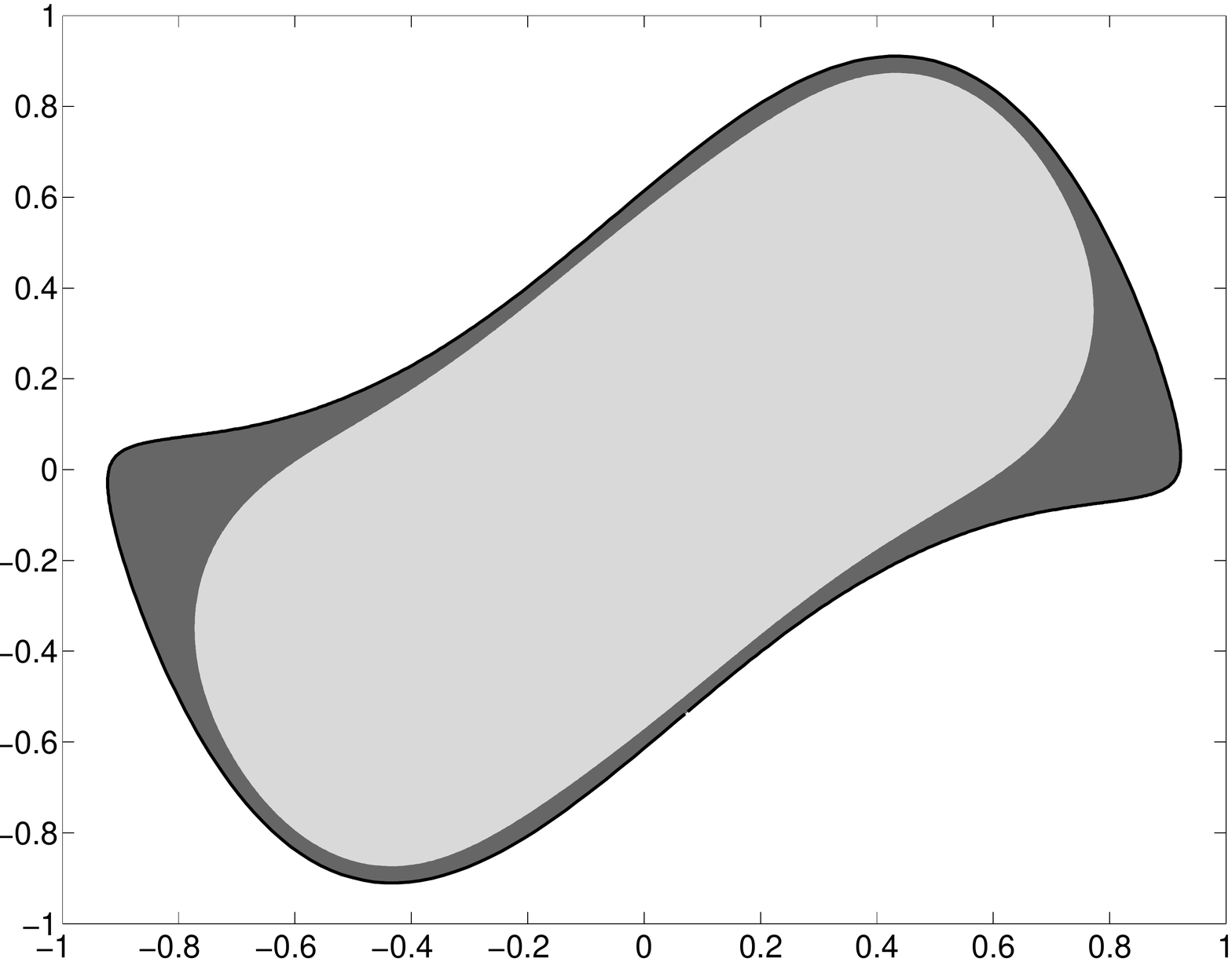}} 
	\put(250,200){\includegraphics[width=70mm]{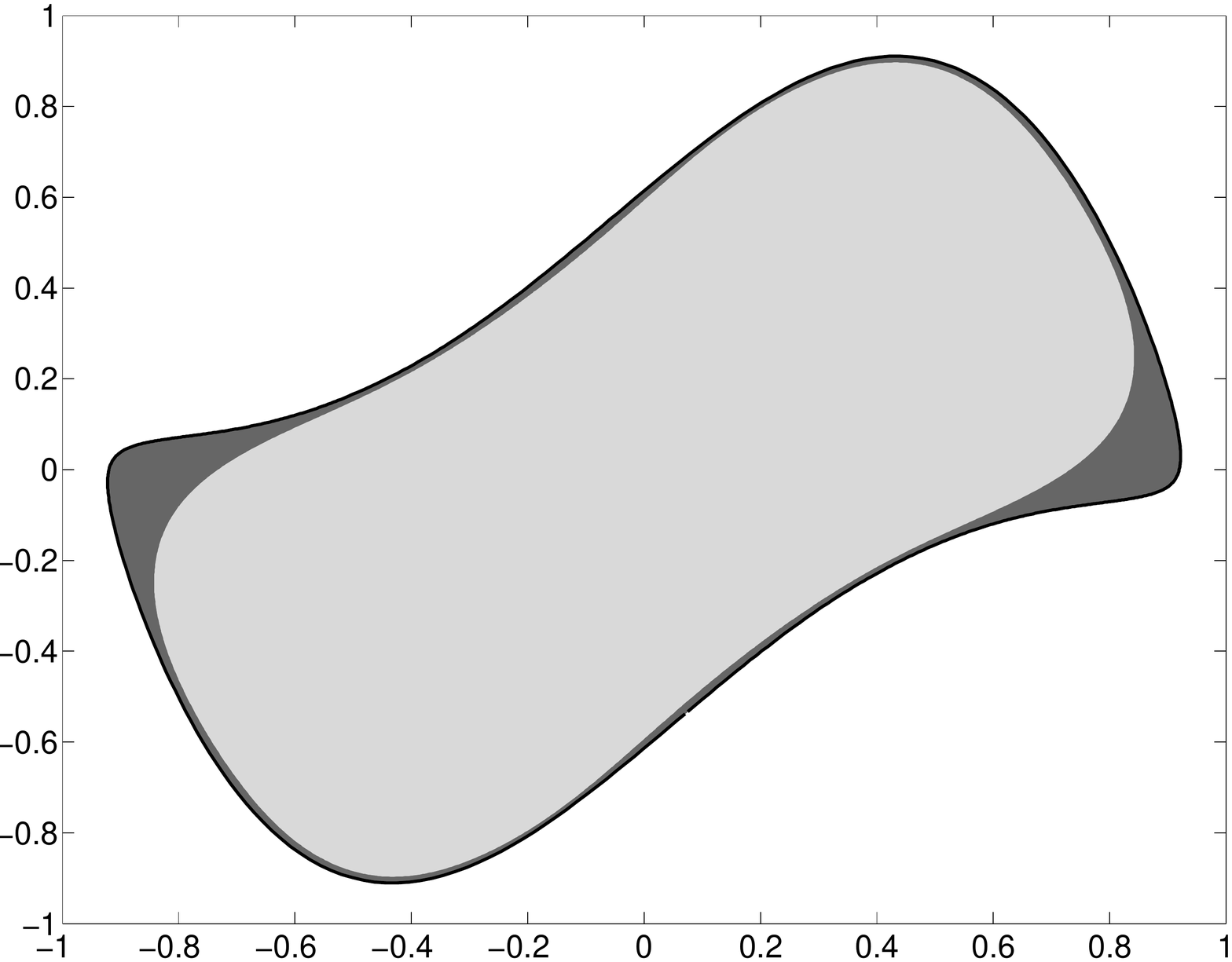}}
	\put(20,20){\includegraphics[width=70mm]{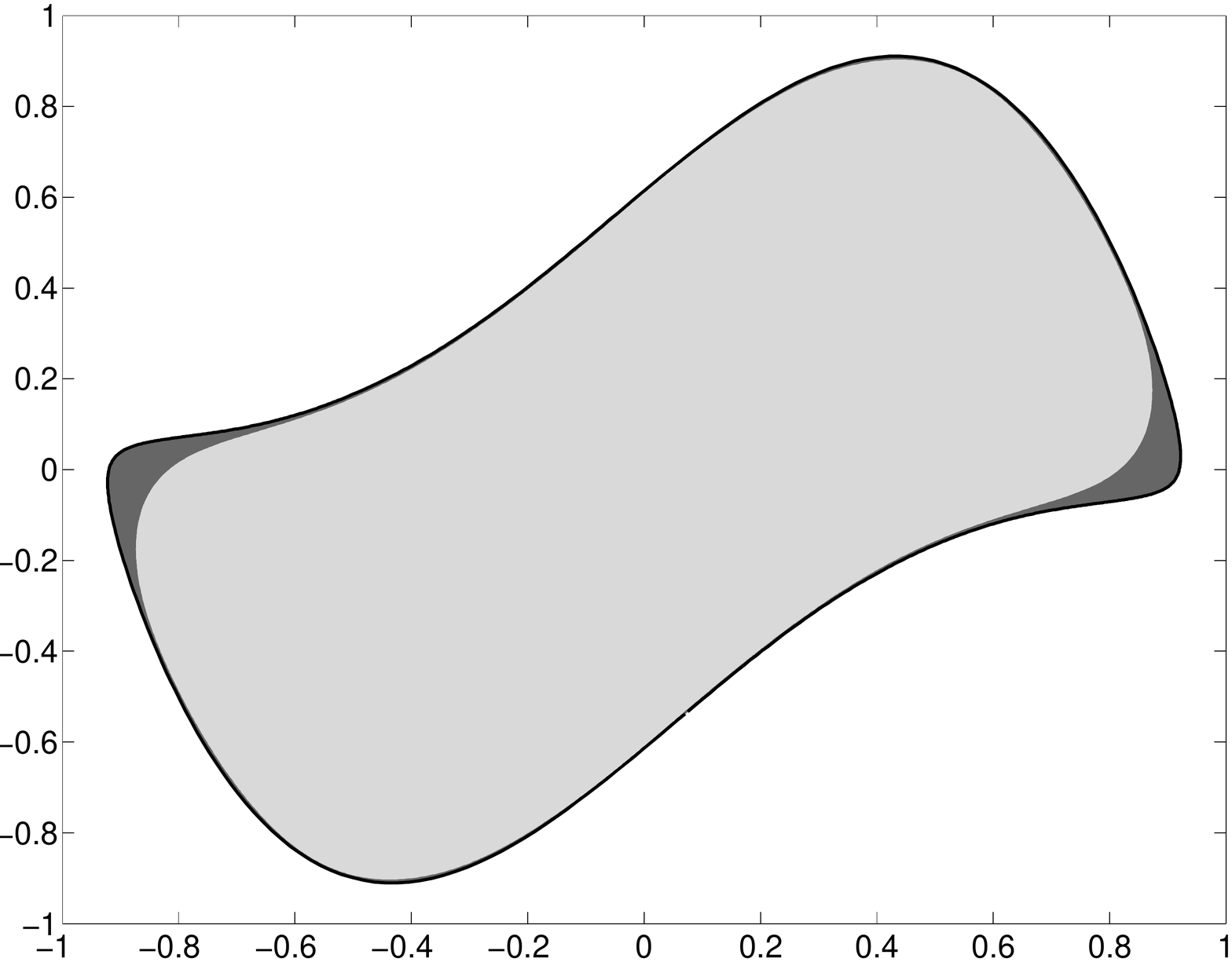}}
	\put(250,20){\includegraphics[width=70mm]{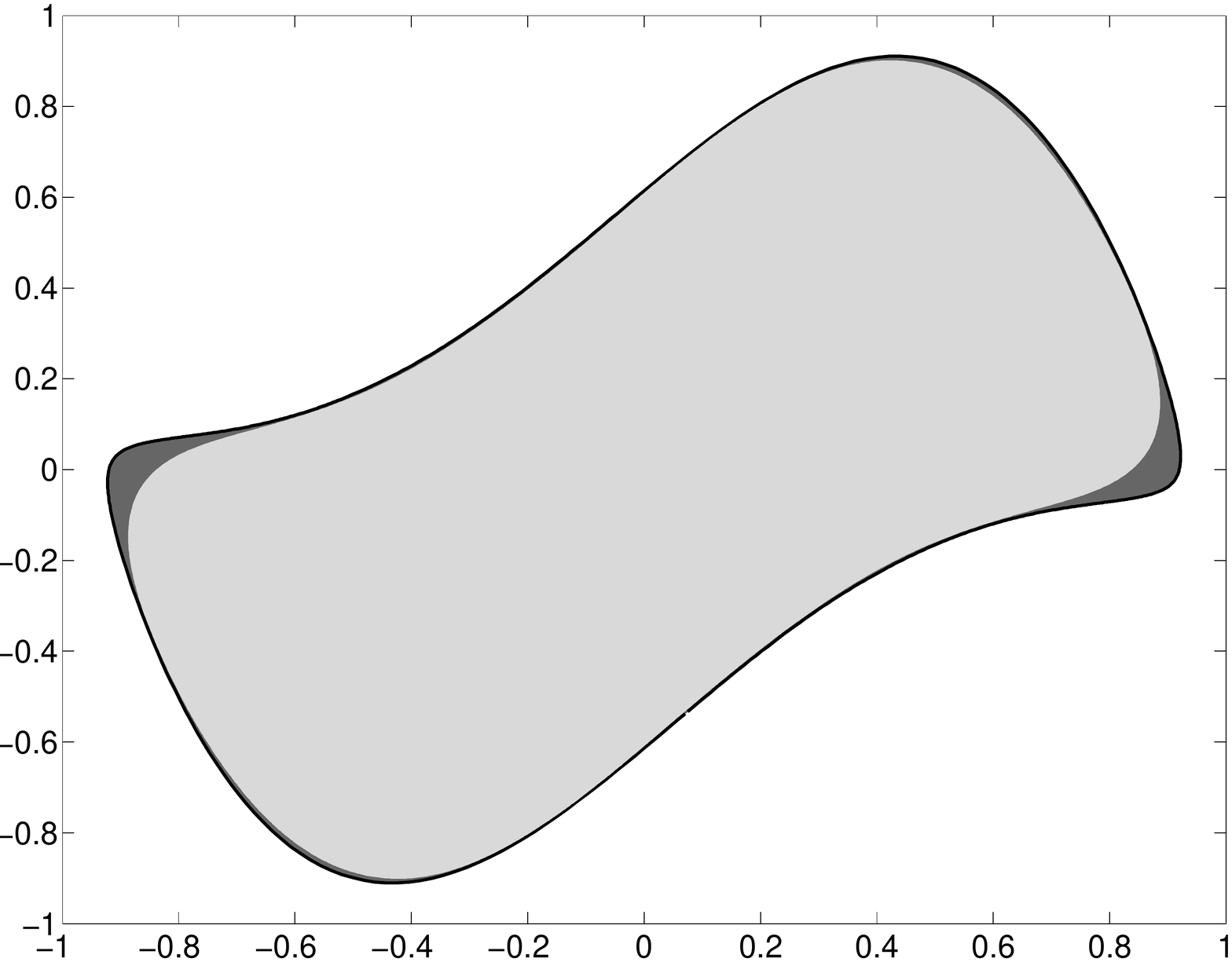}}

	\put(120,190){\footnotesize $x$}
	\put(350,190){\footnotesize $x$}
	\put(120,10){\footnotesize $x$}
	\put(350,10){\footnotesize $x$}

	\put(45,330){\footnotesize $d = 9$}
	\put(275,330){\footnotesize $d=12$}
	\put(45,150){\footnotesize $d = 15$}
	\put(275,150){\footnotesize $d=18$}
	\end{picture}
	\caption{Van der Pol oscillator -- polynomial inner approximations (light gray) to the ROA (dark gray) for degrees $d\in\{9,12,15,18\}$.}
	\label{fig:2}
\end{figure*}

\begin{figure}[!h]
 \centering
 \subfigure[Univariate cubic dynamics -- constraint set $X = {[}\!-\!0.7,0.7{]}$, $\mathrm{deg}\,w = 6$ ($\mathrm{deg}\,v = 16$).Volume approximation error 2.25\,\%.  ]{
  \includegraphics[width=70mm]{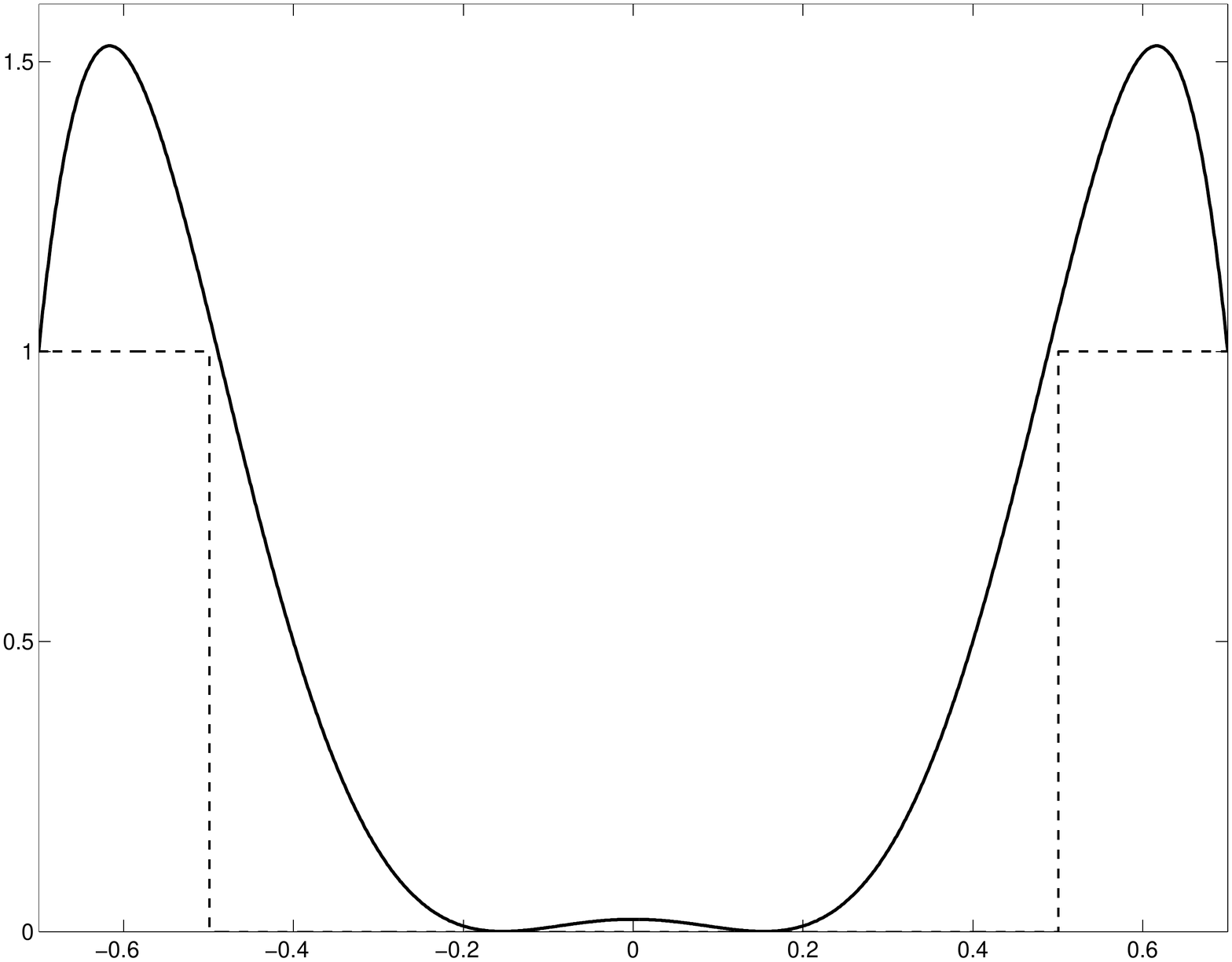}
   }\quad
 \subfigure[Van der Pol oscillator -- $\mathrm{deg}\,w = 8$ ($\mathrm{deg}\,v = 18$). Volume approximation error 5.46\,\%.]{
  \includegraphics[width=70mm]{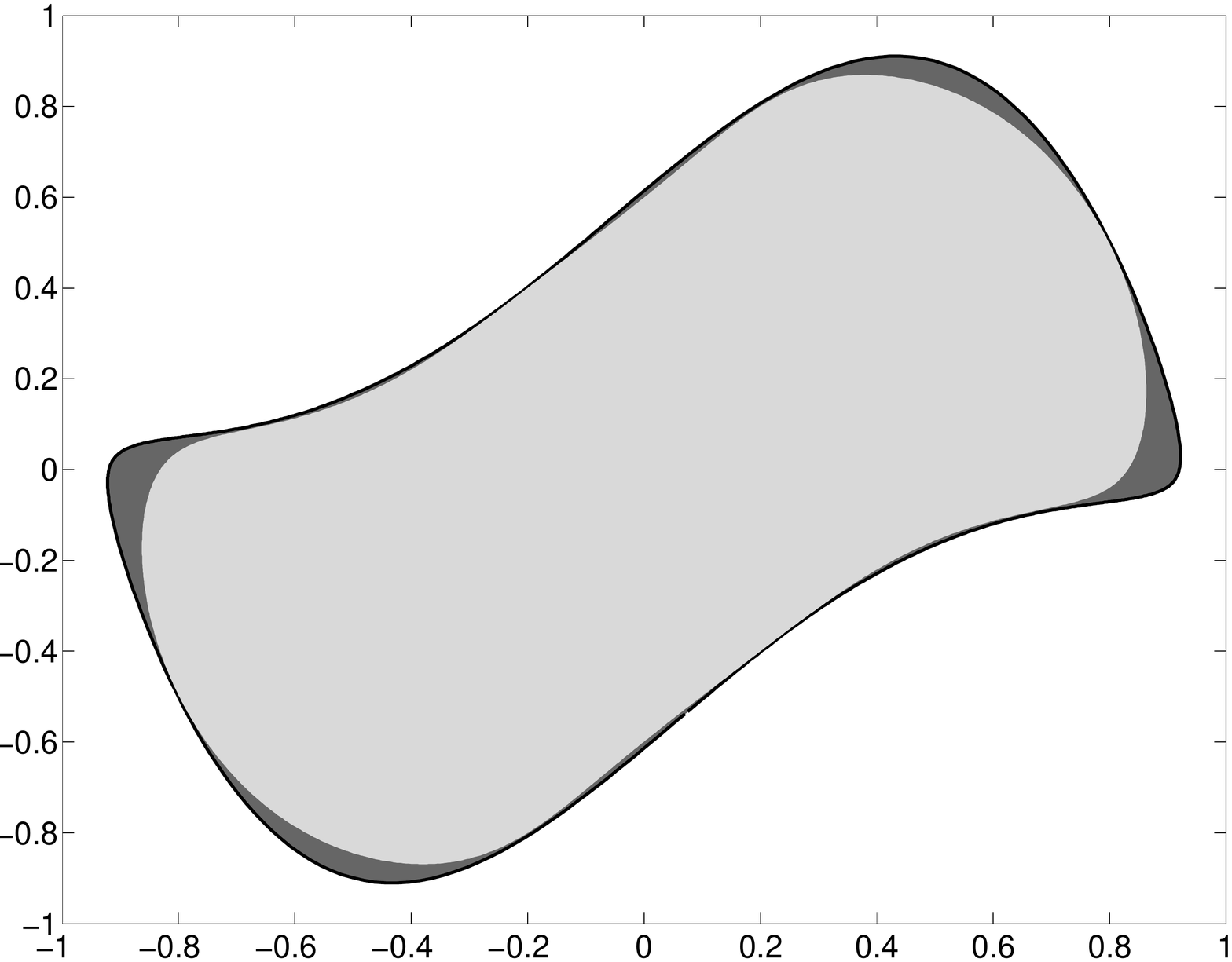}
   }
 \caption{Low order approximations to the ROA. Left: tighter constraint set $X$ and low order $w$ (compare with~Figure~\ref{fig:1}). Right: low order $w$ only (compare with~Figure~\ref{fig:2}).}
 \label{fig:lowOrde}
\end{figure}

As a second example consider a scaled\footnote{The coefficients were chosen so that the ROA
fits within the box $[-1,1]^2$.} version of the uncontrolled reversed-time Van~der~Pol oscillator given by
\begin{align*}
	\dot{x}_1 &= -2.0x_2,\\
	\dot{x}_2 &= 0.80x_1 + 10(x_1^2-0.21)x_2.
\end{align*}
We take $T=1$ and $X_T = \{x\in\mathbb{R}^n\: : \: \|x\|_2\le 0.50\}$ and $X:=\{x\in\mathbb{R}^n \::\: \| x\|\le 1.1\} $. The ROA is bounded, having the characteristic Van~der~Pol shape. Plots of polynomial sublevel set approximations of degrees $d\in \{9,12,15,18\}$ are shown in Figure~\ref{fig:2}. We observe a relatively fast convergence to the ROA, which is also documented by the relative volume errors reported in Table~\ref{tab:2}. Figure~\ref{fig:vp3D} then shows a degree 18 polynomial approximation to the indicator function of the complement ROA.

\begin{table}[!ht]
\centering
\caption{\rm \small Van der Pol oscillator -- relative error of the inner approximation to the ROA $X_0$ as a function of the approximating polynomial degree.} \vspace{1mm}
\begin{tabular}{c|cccc}\label{tab:2}
degree & 9 & 12 & 15 & 18 \\\hline
error & 18.3\,\% & 8.4\,\% & 3.8\,\% & 3.1\,\% 
\end{tabular}
\end{table}

\subsection{Low order approximations}\label{sec:lowOrder}
In the examples above, relatively high order polynomials had to be used to obtain tight approximations, which can limit subsequent applicability of the approximations. There are several ways to obtain low order approximations of similar quality. First of all, since the integral of a polynomial $w$ is minimized over the constraint set $X$, it is desirable that $X$ be a good outer approximation of the ROA. Of course, selecting $X$ is possible only if it is an artificially specified outer approximation of the ROA, not a constraint set coming from physical requirements on the system. More importantly, notice that in problem~(\ref{dlmi}) the system dynamics enters the constraints on the polynomial $v(t,x)$, whereas the polynomial $w(x)$ is only upper-bounding $v(t,x)+1$ for $t=0$. Since the inner approximations are given by sublevel sets of $w$, it is possible and plausible to choose different degrees of $w$ and $v$ -- low for $w$ and higher for $v$. Both techniques are illustrated in Figure~\ref{fig:lowOrde}; in Figure~\ref{fig:lowOrde}~(a) we consider the univariate cubic dynamics and we both shrink the constraint set $X$ and choose low order $w$ while keeping $v$ of higher order. In Figure~\ref{fig:lowOrde}~(b) we consider the Van Der Pol oscillator, keeping the  constraint set $X$ unchanged and only selecting low order $w$. The inner approximations obtained are indeed significantly tighter for the given degrees (compare with Figures~\ref{fig:1} and~\ref{fig:2}).

\section{Conclusion}\label{sec:conclusion}
This paper presented an infinite dimensional convex characterization of the region of attraction (ROA) for uncontrolled polynomial systems, following the approach initiated in our previous work \cite{roa}. Finite dimensional dual relaxations yield a converging sequence of inner approximations to the ROA, thereby complementing the outer approximations of \cite{roa}. One of the virtues of the approach is its conceptual simplicity -- the resulting approximation is the outcome of a single SDP or LMI problem with no free parameters except for the relaxation order. The approximations itself are also simple, given by sublevel sets of polynomials of predefined degrees.

Nevertheless, this approach does not escape the curse of dimensionality -- indeed, whereas the number of variables of the LMI relaxations grows polynomially with the relaxation order, this number grows exponentially with the state dimension. Tailored structure-exploiting SDP solvers could enable this approach to reach higher dimensions. In addition, a different choice of basis functions (e.g., Chebyshev polynomials rather than monomials) would improve numerical conditioning of the LMIs, allowing higher oder relaxations to be computed.

Future research directions include inner approximations in a controlled setting and the related problem of robust region of attraction / reachable set computation with either unknown but constant uncertainty or a time-varying disturbance. The cases of asymptotic region of attraction and maximum (controlled) positively invariant set computation are amenable to similar tools.

\end{document}